\documentclass[11pt]{article}
\usepackage{indentfirst}
\usepackage{misccorr}
\usepackage{graphicx}
\usepackage{amsmath}
\usepackage{amsfonts}
\usepackage{amsthm}
\usepackage{bbm}
\usepackage{mathrsfs}
\usepackage{xcolor}

\usepackage{amssymb,amscd}

\newtheorem{theorem}{Theorem}[section]

\newtheorem{lemma}[theorem]{Lemma}

\newtheorem{proposition}[theorem]{Proposition}

\newcommand{\eee}{{\rm e}}

\newcommand{\dod}{\overset{{\rm d}}{\to}}
\newcommand{\me}{\mathbb{E}}
\newcommand{\mn}{\mathbb{N}}
\newcommand{\mr}{\mathbb{R}}
\newcommand{\mmp}{\mathbb{P}}
\DeclareMathOperator{\1}{\mathbbm{1}}

\begin{document}
\title{Limit theorems for random Dirichlet series: boundary case}\date{}
\author{Alexander Iksanov\footnote{Faculty of Computer Science and Cybernetics, Taras Shevchenko National University of Kyiv, Ukraine; e-mail address:
iksan@univ.kiev.ua} \ and \ Ruslan Kostohryz\footnote{Faculty of Computer Science and Cybernetics, Taras Shevchenko National University of Kyiv, Ukraine; e-mail address:
kostogriz2909@gmail.com}}
\maketitle

\begin{abstract}
Buraczewski et al (2023) proved a functional limit theorem (FLT) and a law of the iterated logarithm (LIL) for a random Dirichlet series $\sum_{k\geq 2}(\log k)^\alpha k^{-1/2-s}\eta_k$ as $s\to 0+$, where $\alpha>-1/2$ and $\eta_1$, $\eta_2,\ldots$ are independent identically distributed random variables with zero mean and finite variance. We prove a FLT and a LIL in a boundary case $\alpha=-1/2$. The boundary case is more demanding technically than the case $\alpha>-1/2$.
\end{abstract}


\hphantom{2020 Mathematics Subject Classification: } Secondary: 60G50

\section{Introduction}

Recently there has been a surge of activity around limit theorems for random Dirichlet series and their zeros. Throughout the paper, by {\it random Dirichlet series}, we mean a random series parameterized by $s>0$:
$$X_\alpha(s):=\sum_{k\geq 2}\frac{(\log k)^\alpha}{k^{1/2+s}}\eta_k,$$ where $\alpha\in\mr$ and $\eta_1$, $\eta_2,\ldots$ are independent copies of a random variable $\eta$ with zero mean and finite variance, which live on a probability space $(\Omega, \mathfrak{F}, \mmp)$. By Kolmogorov's three series theorem, the series defining $X_\alpha(s)$ converges almost surely (a.s.) for each $s>0$. Furthermore, if $\alpha<-1/2$, the series $X_\alpha(0+)$ converges a.s. by the same theorem. Thus, as far as limit theorems for $X_\alpha(s)$ as $s\to0+$ are concerned, the case $\alpha<-1/2$ is not interesting, hence excluded in the sequel.

The following functional limit theorem (FLT) and law of the iterated logarithm (LIL) were known in the case $\alpha>-1/2$. We write $\Longrightarrow$ for weak convergence in a function space and $C(0,\infty)$ and $C[0,\infty)$ for the spaces of real-valued continuous functions defined on $(0,\infty)$ and $[0,\infty)$, respectively. It is assumed that the spaces $C(0,\infty)$ and $C[0,\infty)$ are endowed with the topology of locally uniform convergence.
\begin{proposition}\label{prop:flt}
Assume that $\me[\eta]=0$, $\sigma^2:=\me [\eta^2]\in (0,\infty)$ and let $\alpha>-1/2$. Then $$\Big(s^{1/2+\alpha}\sum_{k\geq 2}\frac{(\log k)^\alpha}{k^{1/2+st}}\eta_k\Big)_{t>0}~\Longrightarrow~ \Big(\sigma\int_{[0,\infty)}y^\alpha\eee^{-ty}{\rm d}B(y)\Big)_{t>0},\quad s\to 0+$$ on $C(0,\infty)$, where $(B(y))_{y\geq 0}$ is a standard Brownian motion.
\end{proposition}

Proposition \ref{prop:flt} follows from Corollary 2.3 in \cite{Buraczewski etal:2023}. In the cited article, the result was derived by a specialization of a FLT for $X_\alpha(s)$, with a complex-valued $\eta$, in the space of analytic functions.

For a family $(x_t)$ of real numbers denote by $C((x_t))$ the set of its limit points.
\begin{proposition}\label{prop:lil}
Assume that $\me[\eta]=0$, $\sigma^2=\me [\eta^2]\in (0,\infty)$ and let $\alpha>-1/2$. Then, almost surely,
\begin{equation*}\label{eq:limset}
C\bigg(\bigg(\Big(\frac{2^{2\alpha}}{\sigma^2 \Gamma(1+2\alpha)}\frac{s^{1+2\alpha}}{\log\log 1/s}\Big)^{1/2}\sum_{k\geq 2}\frac{(\log k)^\alpha}{k^{1/2+s}}\eta_k: s\in (0, \eee^{-1})\bigg)\bigg)=[-1,1],
\end{equation*}
where $\Gamma$ is the Euler gamma function.
\end{proposition}

Proposition \ref{prop:lil} can be found in Theorem 3.1 of \cite{Buraczewski etal:2023}. A classical form of the LIL in terms of $\limsup$ and $\liminf$ was earlier obtained in Theorem 1.1 of \cite{Aymone+Frometa+Misturini:2020} in a rather particular case $\mmp\{\eta=\pm 1\}=1/2$ and $\alpha=0$. Nevertheless, we stress that the work \cite{Aymone+Frometa+Misturini:2020} gave impetus to both \cite{Buraczewski etal:2023} and the present paper.

Our purpose is to formulate and prove counterparts of Propositions \ref{prop:flt} and \ref{prop:lil} in a boundary case $\alpha=-1/2$.

In the case $\alpha=0$, real zeros of random Dirichlet series have been in focus of the recent papers \cite{Aymone:2019, Aymone+Frometa+Misturini:2024, Zhao+Huang:2024} (it is assumed in \cite{Zhao+Huang:2024} that the distribution of $\eta$ is symmetric $\gamma$-stable for $\gamma\in (0,2]$). In the case $\alpha>-1/2$, limit theorems for complex and real zeros of $s\mapsto X_\alpha(s)$ were proved in \cite{Buraczewski etal:2023}. Although we do not directly investigate zeros of $s\mapsto X_{-1/2}(s)$ in the present paper, our LIL stated in Theorem \ref{main} below entails that, a.s., $s\mapsto X_{-1/2}(s)$ has infinitely many real zeros in any right neighborhood of $0$.

\section{Main results}

We start by stating a FLT for $X_{-1/2}(s)$, properly scaled, as $s\to 0+$. If $\alpha>-1/2$, the variance of $X_\alpha(s)$ exhibits a polynomial growth, whereas the growth of ${\rm Var}\,[X_{-1/2}(s)]$ is logarithmic. This partly justifies the facts that the scaling of time in Proposition \ref{prop:flt} is $st$, that is, standard, whereas the scaling of time in Theorem \ref{thm:flt} is $s^t$.
\begin{theorem}\label{thm:flt}
Assume that $\me[\eta]=0$ and $\sigma^2=\me [\eta^2]\in (0,\infty)$. Then $$\Big(\frac{1}{(\log 1/s)^{1/2}}\sum_{k\geq 2}\frac{(\log k)^{-1/2}}{k^{1/2+s^t}}\eta_k\Big)_{t\geq 0}~\Longrightarrow~ (\sigma B(t))_{t\geq 0},\quad s\to 0+$$ on $C[0,\infty)$, where $(B(t))_{t\geq 0}$ is a standard Brownian motion.
\end{theorem}
Observe that the limit process in Theorem \ref{thm:flt} is nowhere differentiable a.s., whereas the limit process in Proposition \ref{prop:flt} is infinitely differentiable a.s. This distinction is a manifestation of intricacy of the boundary case $\alpha=-1/2$.

We proceed with a LIL for $X_{-1/2}(s)$ as $s\to 0+$. The FLT given in Theorem \ref{thm:flt} was used for guessing the LIL's form, namely, the factor $\log\log\log 1/s$ in the normalization.
\begin{theorem}\label{main}
Assume that $\me[\eta]=0$ and $\sigma^2=\me[\eta^2]\in(0,\infty)$. Then
\begin{equation}\label{3.3}
C\Big(\Big(\Big(\frac{1}{2\sigma^2\log 1/s\, \log \log\log 1/s}\Big)^{1/2}\sum_{k\ge 2}\frac{(\log k)^{-1/2}}{k^{1/2+s}}\eta_k : s\in(0, \eee^{-\eee})\Big)\Big)=[-1,1]\quad\text{\rm a.s.}
\end{equation}
In particular,
\begin{equation}\label{3.1}
\limsup_{s\to 0+}\Big(\frac{1}{\log 1/s\, \log\log\log 1/s}\Big)^{1/2}\sum_{k\ge 2}\frac{(\log k)^{-1/2}}{k^{1/2+s}}\eta_k=(2\sigma^2)^{1/2}\quad\text{\rm a.s.}
\end{equation}
and
\begin{equation}\label{3.2}
\liminf_{s\to 0+}\Big(\frac{1}{\log 1/s\, \log\log\log 1/s}\Big)^{1/2}\sum_{k\ge 2}\frac{(\log k)^{-1/2}}{k^{1/2+s}}\eta_k=-(2\sigma^2)^{1/2}\quad\text{\rm a.s.}
\end{equation}
\end{theorem}

Since ${\rm Var}\,[X_{-1/2}(s)]\sim \sigma^2 \log 1/s$ as $s\to 0+$, we infer $\log\log {\rm Var}\,[X_{-1/2}(s)]\sim \log\log\log 1/s$ as $s\to 0+$, where, as usual, $f(s)\sim g(s)$ as $s\to 0+$ means that $\lim_{s\to 0+}(f(s)/g(s))=1$. Thus, formulas \eqref{3.1} and \eqref{3.2} are indeed laws of the {\it iterated} logarithm.

The remainder of the paper is structured as follows. Theorems \ref{main} and  \ref{thm:flt} are proved in Sections \ref{sect:aux} and \ref{sect:flt}, respectively. The reversed order of the proofs is necessitated by the fact that our proof of Theorem \ref{thm:flt} uses some arguments and calculations from the proof of Theorem \ref{main}. At the first glance it looks plausible that an economical  proof of the LIL may be based on a strong approximation by a Brownian motion of a standard random walk generated by $\eta$. In Section \ref{sect:failure} we explain that this naive idea fails.

\section{Proof of Theorem \ref{main}}\label{sect:aux}

It is convenient to split the presentation into two pieces. We start by proving one half of Theorem \ref{main}. In what follows we write $\log^{(3)}$ for $\log\log\log$.
\begin{proposition}\label{pr1}
Under the assumptions of Theorem \ref{main},
\begin{equation}\label{5.1}
\limsup_{s\to 0+}\Big(\frac{1}{\log 1/s\, \log^{(3)} 1/s}\Big)^{1/2}\sum_{k\ge 2}\frac{(\log k)^{-1/2}}{k^{1/2+s}}\eta_k\leq(2\sigma^2)^{1/2}\quad\text{\rm a.s.}
\end{equation}
and
\begin{equation}\label{5.2}
\liminf_{s\to 0+}\Big(\frac{1}{\log 1/s\, \log^{(3)}1/s}\Big)^{1/2}\sum_{k\ge 2}\frac{(\log k)^{-1/2}}{k^{1/2+s}}\eta_k\geq-(2\sigma^2)^{1/2}\quad\text{\rm a.s.}
\end{equation}
\end{proposition}

Replacing $\eta_k$ with $\eta_k/\sigma$ we can work under the assumption that $\sigma^2=1$. For $s\in(0, \eee^{-\eee})$, put $$f(s)=\Big(\frac{1}{2\log 1/s\, \log^{(3)}1/s}\Big)^{1/2}.$$
Let $M: (0,\infty)\to \mn_0$ denote a 
function satisfying $\lim_{s\to 0+}M(s)=+\infty$ and
\begin{equation}\label{eq:growth}
\lim_{s\to 0+}\frac{M(s)}{\log 1/s}=0.
\end{equation}
Here, as usual, $\mn_0:=\mn\cup\{0\}$.

We start by removing from the series in focus an initial fragment with a vanishing contribution. In all the lemmas given below we assume without further notice that the assumptions of Theorem \ref{main} are in force.
\begin{lemma}\label{lemma_2}
The following limit relation holds
$$\lim_{s\to0+}f(s)\sum_{k=2}^{M(s)}\frac{(\log k)^{-1/2}}{k^{1/2+s}}\eta_k=0\quad\text{\rm a.s.}$$
\end{lemma}
\begin{proof}
Put $T_0:=0$ and $T_n:=\eta_1+\ldots+\eta_n$ for $n\in\mathbb{N}$. According to the LIL for standard random walks,
\begin{equation}\label{5.3}
|T_n|\leq \max_{k\leq n}|T_k|=O((n\log\log n)^{1/2}),\quad n\to\infty \quad\text{a.s.}
\end{equation}
Observe that
$$\sum_{k=2}^{M(s)}\frac{(\log k)^{-1/2}}{k^{1/2+s}}\eta_k=\int_{(3/2,\,M(s)]}\frac{{\rm d}T_{\lfloor x\rfloor}}{(\log x)^{1/2}x^{1/2+s}}.$$ Integrating by parts we obtain
\begin{multline*}
\int_{(3/2,\, M(s)]}\frac{(\log x)^{-1/2}}{x^{1/2+s}} {\rm d}T_{\lfloor x\rfloor }=\frac{(\log M(s))^{-1/2}T_{M(s)}}{(M(s))^{1/2+s}}\\-\frac{(\log 3/2)^{-1/2}\eta_1}{(3/2)^{1/2+s}}+ 
\int_{3/2}^{M(s)}\frac{((\log x)^{-3/2}/2+(1/2+s)(\log x)^{-1/2})T_{\lfloor x\rfloor}}{x^{3/2+s}}{\rm d}x.
\end{multline*}
Relation \eqref{eq:growth} entails
$\lim_{s\to 0+}(M(s))^s=1$. This in combination with \eqref{5.3} enables us to conclude that,
as $s\to 0+$,
\begin{multline*}
\frac{(\log M(s))^{-1/2}|T_{M(s)}|}{(M(s))^{1/2+s}}~\sim~ \frac{(\log M(s))^{-1/2}|T_{M(s)}|}{(M(s))^{1/2}}=O((\log M(s))^{-1/2}(\log\log M(s))^{1/2}) = o(1).
\end{multline*}
Since $\lim_{s\to 0+}f(s)=0$, the latter ensures that
\begin{equation*}
\lim_{s\to 0+} f(s)\frac{(\log M(s))^{-1/2}|T_{M(s)}|}{(M(s))^{1/2+s}}= 0 \quad\text{a.s.}
\end{equation*}
To complete the proof, it is sufficient to show that $$\lim_{s\to 0+}f(s)\int_{3/2}^{M(s)}\frac{(\log x)^{-1/2}|T_{\lfloor x\rfloor}|}{x^{3/2+s}}{\rm d}x=0\quad\text{a.s.}$$ To this end, write
\begin{multline*}
\int_{3/2}^{M(s)}\frac{(\log x)^{-1/2}|T_{\lfloor x\rfloor}|}{x^{3/2+s}}{\rm d}x\leq \big(\max_{k\leq M(s)}\,|T_k|\big) \int_{3/2}^{M(s)}\frac{(\log x)^{-1/2}}{x^{3/2+s}}{\rm d}x\\=O((M(s)\log\log M(s))^{1/2})O(1)=O((M(s)\log\log M(s))^{1/2}),\quad s\to 0+\quad\text{a.s.}
\end{multline*}
having utilized \eqref{5.3} for the penultimate equality. Since \eqref{eq:growth} entails $$\lim_{s\to 0+}\frac{M(s)\log\log M(s)}{\log 1/s\,\log^{(3)}1/s}=0,$$ the claim follows.
\end{proof}

Put $g(s):=\mathbb{E}\Big(\sum_{k\geq 2}\frac{(\log k)^{-1/2}}{k^{1/2+s}}\eta_k\Big)^2=\sum_{k\geq 2}\frac{(\log k)^{-1}}{k^{1+2s}}$ for $s>0$ and recall that $g(s)\sim \log 1/s$ as $s\to 0+$. For $k\in \mathbb{N}$, $\rho>0$ and $s\in(0, \eee^{-\eee})$, define the event
\begin{equation*}
A_{k,\rho}(s):=\Big\{|\eta_k|>\frac{\rho}{\log\log 1/s}\Big(\frac{k^{1+2s}\log k\, g(s)}{\log^{(3)}1/s}\Big)^{1/2}\Big\}.
\end{equation*}

Next, we demonstrate that the second (remaining) part of the series also vanishes if the variables $\eta_k$ are properly truncated.
\begin{lemma}\label{lemma_3}
For all $\rho>0$,
\begin{equation}\label{lem4.1}
\lim_{s\to 0+}\sum_{k\ge M(s)+1}\frac{(\log k)^{-1/2}}{k^{1/2+s}}|\eta_k|\1_{A_{k,\rho}(s)}=0\quad \text{\rm{a.s.}}
    \end{equation}
    and
    \begin{equation}\label{lem4.2}
        \lim_{s\to 0+}
        \sum_{k\ge M(s)+1}\frac{(\log k)^{-1/2}}{k^{1/2+s}}\mathbb{E}\big(|\eta_k|\1_{A_{k,\rho}(s)}\big)=0.
    \end{equation}
\end{lemma}
\begin{proof}
Put $h(s):=(\log\log 1/s)(\log^{(3)}1/s)^{1/2}$. Observe that, for $k\geq 3$ and $s\geq 0$, $k^{2s} 
\log k\geq 1$. Hence, for $s\in (0, \eee^{-\eee})$,
\begin{equation*}
\sum_{k\ge M(s)+1}\frac{(\log k)^{-1/2}}{k^{1/2+s}}|\eta_k|\1_{A_{k,\rho}(s)}\le \sum_{k\ge M(s)+1}\frac{|\eta_k|}{k^{1/2}}\1_{\{k^{-1/2}|\eta_k|>\rho(g(s))^{1/2} (h(s))^{-1}\}}\quad\text {{\rm a.s.}}
\end{equation*}
The assumption $\mathbb{E}[\eta^2]<\infty$ entails $\lim_{k\to\infty}k^{-1/2}|\eta_k|=0$ a.s. and thereupon $\sup_{k\ge 1}(k^{-1/2}|\eta_k|)<\infty$ a.s. Since $\lim_{s\to 0+}((g(s)
)^{1/2} (h(s))^{-1})=+\infty$, we infer
\begin{equation*}
\1_{\{k^{-1/2}|\eta_k|>\rho(g(s))^{1/2} (h(s))^{-1}\}}\le \1_{\{\sup_{k\ge 1}\,(k^{-1/2}|\eta_k|)>\rho(g(s))^{1/2} (h(s))^{-1}\}}=0
\end{equation*}
a.s.\ for small $s$. We have proved that the sum in \eqref{lem4.1} is equal to $0$ a.s.\ for small enough $s$.

Relation \eqref{lem4.2} is justified as follows:
\begin{multline*}
\sum_{k\ge M(s)+1}\frac{(\log k)^{-1/2}}{k^{1/2+s}}\mathbb{E}\big(|\eta_k|\1_{A_{k,\rho}(s)}\big) \le  \sum_{k\ge M(s)+1}k^{-1/2}\mathbb{E}\big(|\eta|\1_{\{\rho^{-1}(g(s))^{-1/2} h(s) 
|\eta|>k^{1/2}\}}\big)\\ \le \mathbb{E}\Bigg(|\eta|\sum_{k=1}^{\lfloor \rho^{-2}(g(s)
)^{-1}(h(s))^2\eta^2\rfloor}k^{-1/2}\Bigg)  \le 2\rho^{-1}\mathbb{E}\eta^2(g(s)
)^{-1/2} h(s)\\=2\rho^{-1}(g(s))^{-1/2}h(s)~\to~0,\quad s\to 0+.
\end{multline*}
The proof of Lemma \ref{lemma_3} is complete.
\end{proof}
In what follows, $(A_{k,\rho}(s))^c$ denotes the complement of $A_{k,\rho}(s)$, that is, for $k\in\mathbb{N}, \rho>0$ and $s\in(0, \eee^{-\eee})$,
\begin{equation*}
(A_{k,\rho}(s))^c:=\Big\{|\eta_k|\le\frac{\rho}{\log\log1/s}\Big(\frac{k^{1+2s}\log k\, g(s)}{\log^{(3)}1/s}\Big)^{1/2}\Big\}.
\end{equation*}
Our next result is concerned with the fragment of the series giving the principal contribution. However, this is a light version of what we really need, for the convergence here is only along a sequence.
\begin{lemma}\label{lemma4}
Fix any $\gamma\in (0,(\sqrt{5}-1)/2)$,
pick any $\rho=\rho(\gamma)$ satisfying
\begin{equation}\label{5.7}
(1-\gamma)(1+\gamma)^2(2-\exp(2\sqrt{2}(1+\gamma)\rho))>1
\end{equation}
and put $s_n:=\exp(-\exp(n^{1-\gamma}))$ for $n\in\mathbb{N}$. Then
\begin{equation*}
\limsup_{n\to\infty} f(s_n)\sum_{k\ge M(s_n)+1}\frac{(\log k)^{-1/2}\widetilde{\eta}_{k,\rho}(s_n)}{k^{1/2+s_n}}\le 1+\gamma \quad\text{{\rm a.s.},}
\end{equation*}
where $\widetilde{\eta}_{k,\rho}(s):=\eta_k \1_{(A_{k,\rho}(s))^c}-\mathbb{E}\big(\eta_k\1_{(A_{k,\rho}(s))^c}\big)$ for $k\in\mathbb{N}$ and $s\in(0, \eee^{-\eee})$.
\end{lemma}
\begin{proof}
Since $(1-\gamma)(1+\gamma)^2>1$ whenever $\gamma\in (0, (\sqrt{5}-1)/2)$, $\rho$ satisfying \eqref{5.7} does indeed exist.

Put $f^\ast(s):=(2g(s)\,\log^{(3)}1/s)^{-1/2}$ for $s\in (0, \eee^{-\eee})$. Since $f^\ast(s)\sim f(s)$ as $s\to 0+$, we can and do prove the result, with $f^\ast$ replacing $f$. Put
\begin{equation*}
X(s)=f^\ast(s)\sum_{k\ge M(s)+1}\frac{(\log k)^{-1/2}\widetilde{\eta}_{k,\rho}(s)}{k^{1/2+s}},\quad s\in(0, \eee^{-\eee}).
\end{equation*}
Using $\eee^x\le 1+x+(x^2/2)\eee^{|x|}$ for $x\in\mathbb{R}$ and $\mathbb{E}[\widetilde{\eta}_{k,\rho}(s)]=0$ we deduce, for $u\in\mathbb{R}$,
\begin{multline*}
\mathbb{E}[\eee^{uX(s)}]=\prod_{k\ge M(s)+1}\mathbb{E}\Big[\exp\Big(u f^\ast(s)
\frac{(\log k)^{-1/2}\widetilde{\eta}_{k,\rho}(s)}{k^{1/2+s}}\Big)\Big] \\ \le \prod_{k\ge M(s)+1}\Big(1+\frac{u^2(f^\ast(s)
)^2}{2}\frac{(\log k)^{-1}}{k^{1+2s}}\mathbb{E}\Big[(\widetilde{\eta}_{k,\rho}(s))^2\exp\Big(|u|f^\ast(s)
\frac{(\log k)^{-1/2}|\widetilde{\eta}_{k,\rho}(s)|}{k^{1/2+s}}\Big)\Big]\Big).
\end{multline*}
The inequality
\begin{multline}\label{5.8}
|\widetilde{\eta}_{k,\rho}(s)|\le|\eta_k|\1_{(A_{k,\rho}(s))^c}+\mathbb{E}\big(|\eta_k|\1_{(A_{k,\rho}(s))^c}\big)\\\le \frac{2\rho k^{1/2+s}}{\log\log 1/s}\Big(\frac{\log k\, g(s)}{\log^{(3)}1/s}\Big)^{1/2}\le 2\rho k^{1/2+s} \Big(\frac{\log k\, g(s)}{\log^{(3)}1/s}\Big)^{1/2}\quad\text{a.s.,}
\end{multline}
which is valid for integer $k\geq 2$ and $s\in(0, \eee^{-\eee})$, implies that
\begin{equation*}
\exp\Big(|u|f^\ast(s)\frac{(\log k)^{-1/2}|\widetilde{\eta}_{k,\rho}(s)|}{k^{1/2+s}}\Big)\le\exp\Big(\frac{\sqrt{2}\rho|u|}{\log^{(3)}1/s}\Big)\quad\text{a.s.}
\end{equation*}
Together with the inequalities $\mathbb{E}[(\widetilde{\eta}_{k,\rho}(s))^2]\le 1$ and $1+x\le \eee^x$ for $x\in\mathbb{R}$ this gives,
for $u\in\mathbb{R}$,
\begin{multline}\label{5.9}
\mathbb{E}[\eee^{uX(s)}]\le\prod_{k\ge M(s)+1}\exp\Big(\frac{u^2(f^\ast(s))^2}{2}\frac{(\log k)^{-1}}{k^{1+2s}}\exp\Big(\frac{\sqrt{2}\rho|u|}{\log^{(3)}1/s}\Big)\Big)\\
\le\exp\Big(\frac{u^2}{4 \log^{(3)}1/s}\exp\Big(\frac{\sqrt{2}\rho|u|}{\log^{(3)}1/s}\Big)\Big).
\end{multline}
An application of Markov's inequality yields, for $u\ge 0$,
\begin{multline*}
\mathbb{P}\{X(s_n)>1+\gamma\}\le \eee^{-(1+\gamma)u} \mathbb{E}[\eee^{u X(s_n)}]\le \exp\Big(-(1+\gamma)u+\frac{u^2}{4 \log^{(3)}1/s_n}\exp\Big(\frac{\sqrt{2}\rho u}{ \log^{(3)}1/s_n}\Big)\Big).
\end{multline*}
Putting $u=2(1+\gamma)\log^{(3)}1/s_n$ we conclude that
\begin{multline*}
\mathbb{P}\{X(s_n)>1+\gamma\}\le \exp(-(1+\gamma)^2(2-\exp(2\sqrt{2}(1+\gamma)\rho)) \log^{(3)} 1/s_n)\\ =\frac{1}{n^{(1-\gamma)(1+\gamma)^2(2-\exp(2\sqrt{2}(1+\gamma)\rho))}}.
\end{multline*}
Thus, in view of \eqref{5.7}, $\sum_{n\ge 1}\mathbb{P}\{X(s_n)>1+\gamma\}<\infty$, and invoking the direct part of the Borel-Cantelli lemma completes the proof of Lemma \ref{lemma4}.
\end{proof}

Now we present our final, and the most involved, preparatory result. It shows that the convergence along a sequence discussed in Lemma \ref{lemma4} can be lifted to the full convergence along the real numbers.
\begin{lemma}\label{lemma5}
Let $(s_n)_{n\in\mathbb{N}}$ be as defined in Lemma \ref{lemma4}, where $\gamma\in (0, 1/2)$, and $M(s)=\lfloor \log 1/s/\log\log 1/s\rfloor$ for $s\in (0,\eee^{-\eee})$. For $s\in [s_{n+1},s_n]$, the following limit relation holds
\begin{equation*}
\lim_{n\to\infty}f(s)\Big(\sum_{k\ge M(s)+1}\frac{(\log k)^{-1/2}}{k^{1/2+s}}\eta_k -\sum_{k\ge M(s_{n+1})+1}\frac{(\log k)^{-1/2}}{k^{1/2+s_{n+1}}}\eta_k\Big)=0\quad\text{\rm a.s.}
\end{equation*}
\end{lemma}
\begin{proof}
Using the fact that $M$ is a nonincreasing function for the arguments close to $0$, write, for $s\in [s_{n+1},s_n]$,
\begin{multline*}
\sum_{k\ge M(s)+1}\frac{(\log k)^{-1/2}\eta_k}{k^{1/2+s}}-\sum_{k\ge M(s_{n+1})+1}\frac{(\log k)^{-1/2}\eta_k}{k^{1/2+s_{n+1}}}=\sum_{k= M(s)+1}^{M(s_{n+1})}\frac{(\log k)^{-1/2}\eta_k }{k^{1/2+s}}\\+\sum_{k\ge M(s_{n+1})+1}(\log k)^{-1/2}\Big(\frac{1}{k^{1/2+s}}-\frac{1}{k^{1/2+s_{n+1}}}\Big)\eta_k=:I_{n,1}(s)+I_{n,2}(s).
\end{multline*}

\noindent {\sc Analysis of $I_{n,1}(s)$.} Actually, we shall prove that $\lim_{n\to\infty}\sup_{s\in[s_{n+1},\,s_n]}|I_{n,1}(s)|=0$ a.s. This relation is even more than we need because $\lim_{s\to 0+} f(s)=0$. We obtain with the help of summation by parts
\begin{multline*}
I_{n,1}(s)=\frac{(\log M(s_{n+1}))^{-1/2}T_{M(s_{n+1})}}{(M(s_{n+1}))^{1/2+s}}-\frac{(\log (M(s)+1))^{-1/2}T_{M(s)}}{(M(s)+1)^{1/2+s}}\\ +\sum_{k=M(s)+1}^{M(s_{n+1})-1}\Big(\frac{(\log k)^{-1/2}}{k^{1/2+s}}-\frac{(\log (k+1))^{-1/2}}{(k+1)^{1/2+s}}\Big)T_k,
\end{multline*}
where, as before, $T_n=\eta_1+\ldots+\eta_n$ for $n\in\mathbb{N}$. Invoking formula \eqref{5.3} and $\lim_{n\to\infty}(M(s_{n+1}))^{s_{n+1}}=1$ we obtain
\begin{multline*}
\frac{(\log M(s_{n+1}))^{-1/2}|T_{M(s_{n+1})}|}{(M(s_{n+1}))^{1/2+s}}\le \frac{(\log M(s_{n+1}))^{-1/2}|T_{M(s_{n+1})}|}{(M(s_{n+1}))^{1/2+s_{n+1}}}~\sim~ \frac{(\log M(s_{n+1}))^{-1/2}|T_{M(s_{n+1})}|}{(M(s_{n+1}))^{1/2}}
\\=O((\log M(s_{n+1}))^{-1/2}(\log\log M(s_{n+1}))^{1/2})\to 0,\quad n\to\infty\quad\text{\rm a.s.}
\end{multline*}
By a similar argument we conclude that
\begin{equation*}
\lim_{n\to\infty}\sup_{s\in[s_{n+1},\,s_n]}\frac{(\log (M(s)+1))^{-1/2}|T_{M(s)}|}{(M(s)+1)^{1/2+s}}=0\quad\text{\rm a.s.}
\end{equation*}
Further, for $s\in [s_{n+1}, s_n]$ and large $n$,
\begin{multline*}
\Bigg|\sum_{k=M(s)+1}^{M(s_{n+1})-1}\Bigg(\frac{(\log k)^{-1/2}}{k^{1/2+s}}-\frac{(\log (k+1))^{-1/2}}{(k+1)^{1/2+s}}\Bigg)T_k\Bigg|\\ \le
\sum_{k=M(s)+1}^{M(s_{n+1})-1}\Bigg(\frac{(\log k)^{-1/2}}{k^{1/2+s}}-\frac{(\log (k+1))^{-1/2}}{(k+1)^{1/2+s}}\Bigg)|T_k|
\\ \le
\big(\sup_{j\le M(s_{n+1})}|T_j|\big)\sum_{k=M(s)+1}^{M(s_{n+1})-1}\Bigg(\frac{(\log k)^{-1/2}}{k^{1/2+s}}-\frac{(\log (k+1))^{-1/2}}{(k+1)^{1/2+s}}\Bigg)
\\ \le
\big(\sup_{j\le M(s_{n+1})}|T_j|\big)\frac{(\log M(s))^{-1/2}}{(M(s))^{1/2+s}}
\le \big(\sup_{j\le M(s_{n+1})}|T_j|\big)\frac{(\log M(s_n))^{-1/2}}{(M(s_n))^{1/2+s_{n+1}}}
\\ =
O((\log M(s_{n+1}))^{-1/2}(\log\log M(s_{n+1}))^{1/2})\to 0,\quad n\to\infty\quad\text{\rm a.s.}
    \end{multline*}
We have used \eqref{5.3}, $\lim_{n\to\infty}(M(s_{n+1})/M(s_n))=1$ and $\lim_{n\to\infty}(M(s_n))^{s_{n+1}}=1$ for the equality.

\noindent {\sc  Analysis of $I_{n,2}(s)$.} One can show that
\begin{equation}\label{lem4.11}
\lim_{n\to\infty}\sup_{s\geq s_{n+1}}\sum_{k\ge M(s_{n+1})+1}(\log k)^{-1/2}\Big(\frac{1}{k^{1/2+s_{n+1}}}-\frac{1}{k^{1/2+s}}\Big)|\eta_k|\1_{\{|\eta_k|>k^{1/2}\log n\}}=0\quad \text{\rm{a.s.}}
\end{equation}
and
\begin{equation}\label{lem4.21}
\lim_{n\to\infty}\sup_{s\geq s_{n+1}}\sum_{k\ge M(s_{n+1})+1}(\log k)^{-1/2}\Big(\frac{1}{k^{1/2+s_{n+1}}}-\frac{1}{k^{1/2+s}}\Big)\mathbb{E}\big[|\eta_k|\1_{\{|\eta_k|>k^{1/2}\log n\}}\big]=0.
\end{equation}
For instance, relation \eqref{lem4.11} follows from $\sup_{k\ge 1}(k^{-1/2}|\eta_k|)<\infty$ a.s. and
\begin{multline*}
\sum_{k\ge M(s_{n+1})+1}(\log k)^{-1/2}\Big(\frac{1}{k^{1/2+s_{n+1}}}-\frac{1}{k^{1/2+s}}\Big)|\eta_k|\1_{\{|\eta_k|>k^{1/2}\log n\}}\\\leq \sum_{k\ge M(s_{n+1})+1}\frac{|\eta_k|}{k^{1/2}}\1_{\{k^{-1/2}|\eta_k|>\log n\}}\quad\text {{\rm a.s.}}
\end{multline*}
The summands on the right-hand side are equal to $0$ for large enough $n$. Here, we have used $k^{2s_{n+1}}\log k\geq 1$, for $k\geq 3$ and $n\in\mn$. More details can be found in the proof of Lemma \ref{lemma_3}.

Put $\widehat{\eta}_k(n):=\eta_k\1_{\{|\eta_k|\leq k^{1/2}\log n\}}-\me \big[\eta_k\1_{\{|\eta_k|\leq k^{1/2}\log n\}}\big]$ for $k,n\in\mn$. For $n\in\mathbb{N}$ and small positive $u$, put
\begin{equation*}
Y^\ast_n(u):=\sum_{k\ge M(s_{n+1})+1}\frac{(\log k)^{-1/2} \widehat{\eta}_k(n)}{k^{1/2+u}}.
\end{equation*}
In view of \eqref{lem4.11} and \eqref{lem4.21}, it suffices to demonstrate that, for each $s\in[s_{n+1}, s_n]$,
\begin{equation*}
\lim_{n\to\infty}f(s_n)(Y^\ast_n(s)-Y^\ast_n(s_{n+1}))=0\quad\text{\rm a.s.}
\end{equation*}
By a technical reason to be explained below, we shall show that, for each $v\in [v_{n+1}, v_n]$,
\begin{equation}\label{5.11}
\lim_{n\to\infty}f(s_n)(Y_n(v)-Y_n(v_{n+1}))=0\quad\text{\rm a.s.},
\end{equation}
where $Y_n(v):=Y_n^\ast(\exp(-1/v))$, $v_n:=1/(\log 1/s_n)=\exp(-n^{1-\gamma})$ for $n\in\mn$ and $v>0$.

For $j\in\mathbb{N}_0$ and $n\in\mathbb{N}$, put
\begin{equation*}
F_j(n):=\{t_{j,\,m}(n):=v_{n+1}+2^{-j}m(v_n-v_{n+1}):0\le m\le 2^j\}.
\end{equation*}
Note that $F_j(n)\subseteq F_{j+1}(n)$ and put $F(n):=\bigcup_{j\ge 0}F_j(n)$. The set $F(n)$ is dense in the interval $[v_{n+1},v_n]$. For any $u\in[v_{n+1},v_n]$, put
\begin{equation*}
u_j:=\max\{v\in F_j(n): v\le u\}=v_{n+1}+2^{-j}(v_n-v_{n+1})\Big\lfloor\frac{2^j(u-v_{n+1})}{v_n-v_{n+1}}\Big\rfloor.
\end{equation*}
Then $\lim_{j\to\infty} u_j=u$ (we omit the dependence on $n$ in the notation). An important observation is that either $u_{j-1}=u_j$ or $u_{j-1}=u_j-2^{-j}(v_n-v_{n+1})$. Consequently, $u_j=t_{j,m}$ for some $0\le m \le 2^j$, which implies that either $u_{j-1}=t_{j,\,m}$ or $u_{j-1}=t_{j,\,m-1}$. Since $Y_n$ is a.s. continuous on $[v_{n+1}, v_n]$, we obtain
    \begin{multline*}
        |Y_n(u)-Y_n(v_{n+1})|=\lim_{l\to\infty}|Y_n(u_l)-Y_n(v_{n+1})|\\ =\lim_{l\to\infty}\Big|\sum_{j=1}^{l}(Y_n(u_j)-Y_n(u_{j-1}))+Y_n(u_0)-Y_n(v_{n+1})\Big|\\ \le \lim_{l\to\infty}\sum_{j=0}^l\max_{1\le m \le 2^j}|Y_n(t_{j,\,m})-Y_n(t_{j,\,m-1})|=\sum_{j\ge 0}\max_{1\le m \le 2^j}|Y_n(t_{j,\,m})-Y_n(t_{j,\,m-1})|.
    \end{multline*}
Thus, our purpose is to prove that, for all $\varepsilon>0$ and sufficiently large $n_0\in\mathbb{N}$,
\begin{equation*}
\sum_{n\ge n_0}\mathbb{P}\Big\{\sum_{j\ge 0}\max_{1\le m \le 2^j}f(s_n)|Y_n(t_{j,\,m})-Y_n(t_{j,\,m-1})|>\varepsilon\Big\}<\infty.
\end{equation*}
Put $a_j:=(j+1)2^{-j/2}$ for $j\in\mathbb{N}_0$. In view of $\sum_{j\ge 0}a_j<\infty$, it suffices to show that, for all $\varepsilon>0$,
\begin{equation}\label{5.12}
\sum_{n\ge n_0}\sum_{j\ge 0}\mathbb{P}\Big\{\max_{1\le m \le 2^j}f(s_n)|Y_n(t_{j,\,m})-Y_n(t_{j,\,m-1})|>\varepsilon a_j \Big\}<\infty.
\end{equation}

Now we proceed similarly to the proof of Lemma \ref{lemma4} and refer to that proof regarding any missing details. Write, for $u\in\mathbb{R}$ and sufficiently large $n$,
\begin{multline*}
\me\big[\exp\big(\pm u(Y_n(t_{j,\,m})-Y_n(t_{j,\,m-1}))\big)\big]\\=\mathbb{E}\Big[\exp\Big(\pm u \sum_{k\ge M(s_{n+1})+1}\frac{1}{(k\log k)^{1/2}}\Big(\frac{1}{k^{\exp(-1/t_{j,\,m})}}-\frac{1}{k^{\exp(-1/t_{j,\,m-1})}}\Big)\widehat{\eta}_k(n) \Big)\Big] \\ \le \prod_{k\ge M(s_{n+1})+1}\Big( 1+\frac{u^2}{2}\frac{1}{k\log k}\Big(\frac{1}{k^{\exp(-1/t_{j,\,m})}}-\frac{1}{k^{\exp(-1/t_{j,\,m-1})}}\Big)^2\\\times\mathbb{E}\Big[(\widehat{\eta}_k(n))^2\exp\Big(\frac{|u|}{(k\log k)^{1/2}}\Big(\frac{1}{k^{\exp(-1/t_{j,\,m})}}-\frac{1}{k^{\exp(-1/t_{j,\,m-1})}}\Big)|\widehat{\eta}_k(n)|\Big)\Big].
\end{multline*}
Now we prove that, for large $n$, all $j\in\mn_0$ and integer $m\in [0,2^j]$,
\begin{equation*}
A(j,n):=\sum_{k\ge M(s_{n+1})+1}\frac{1}{k\log k}\Big(\frac{1}{k^{\exp(-1/t_{j,\,m})}}-\frac{1}{k^{\exp(-1/t_{j,\,m-1})}}\Big)^2 \leq \frac{2^{-j}(v_n-v_{n+1})}{v_{n+1}^2}.
\end{equation*}
For fixed $a,b>0$, the function $x\mapsto x^{-1}\eee^{-x}(\eee^{-ax}-\eee^{-bx})^2$ is decreasing on $(1,\infty)$. Indeed, $x\mapsto x\eee^{-x}$ is decreasing on $(1,\infty)$, and $$x\mapsto \eee^{-\min(a,b)x} x^{-2}(1-\eee^{-(\max(a,b)-\min(a,b))x})^2$$ is decreasing on $(0,\infty)$ as the product of two positive decreasing functions. As far as monotonicity of the second function is concerned, observe that, up to a multiplicative constant, it is the Laplace-Stieltjes transform of the Lebesgue-Stieltjes convolution of the uniform distribution on $[0,1]$ with itself. Using this argument with $x=\log k$, $a=\exp(-1/t_{j,\,m})$ and $b=\exp(-1/t_{j,\,m-1})$ we conclude that the function of argument $k$ under the sum defining $A(j,n)$ is decreasing, whence
\begin{multline*}
A(j,n)\leq \int_\eee^\infty\frac{1}{x\log x}\Big(\frac{1}{x^{2\exp(-1/t_{j,\,m-1})}}+\frac{1}{x^{2\exp(-1/t_{j,\,m})}}-\frac{2}{x^{\exp(-1/t_{j,\,m-1})+\exp(-1/t_{j,\,m})}}\Big){\rm d}x\\=\int_{2\exp(-1/t_{j,\,m-1})}^\infty \frac{\eee^{-x}}{x}{\rm d}x+\int_{2\exp(-1/t_{j,\,m})}^\infty \frac{\eee^{-x}}{x}{\rm d}x\\-2\int_{\exp(-1/t_{j,\,m-1})+\exp(-1/t_{j,\,m})}^\infty \frac{\eee^{-x}}{x}{\rm d}x=\int_{2\exp(-1/t_{j,\,m-1})}^1 \frac{\eee^{-x}}{x}{\rm d}x\\+\int_{2\exp(-1/t_{j,\,m})}^1 \frac{\eee^{-x}}{x}{\rm d}x-2\int_{\exp(-1/t_{j,\,m-1})+\exp(-1/t_{j,\,m})}^1 \frac{\eee^{-x}}{x}{\rm d}x.
\end{multline*}
Put $I(y):=\int_0^y x^{-1}(1-\eee^{-x}){\rm d}x$ for $y\in [0,1]$ and observe that $$\int_y^1 x^{-1}\eee^{-x}{\rm d}x=-\log y-I(1)+I(y),\quad y\in (0,1].$$ The function $x\mapsto x^{-1}(1-\eee^{-x})$ is decreasing on $(0,\infty)$ as the Laplace-Stieltjes transform of the uniform distribution on $[0,1]$. Hence, $I$ is concave on $[0,1]$ and thereupon, for $a,b\in [0,1]$, $I(2a)+I(2b)-2I(a+b)\leq 0$. As a consequence,
\begin{multline}
A(j,n)\leq -2\log 2+\frac{1}{t_{j,\,m-1}}+\frac{1}{t_{j,\,m}}+2\log\Big(\eee^{-1/t_{j,\,m-1}}+\eee^{-1/t_{j,\,m}}\Big)\\\leq \frac{t_{j,\,m}-t_{j,\,m-1}}{t_{j,\,m-1}t_{j,\,m}}\leq \frac{2^{-j}(v_n-v_{n+1})}{v_{n+1}^2},\label{eq:ajn}
\end{multline}
which proves the claim.

Next we work towards estimating $$B(u,k,n):=\exp\Big(\frac{|u|}{(k\log k)^{1/2}}\Big(\frac{1}{k^{\exp(-1/t_{j,\,m-1})}}-\frac{1}{k^{\exp(-1/t_{j,\,m})}}\Big)|\widehat{\eta}_k(n)|\Big)$$ for $k\geq M(s_{n+1})+1$. Assume first that $2^j\geq v_{n+1}^{-2}(v_n-v_{n+1})$.  Observe that, for $a>0$ and $0<s<t$, $$0\leq \exp(-a\eee^{-1/s})-\exp(-a\eee^{-1/t})\leq a\exp(-a\eee^{-1/s})\eee^{-1/t}(1/s-1/t).$$ Using this inequality with $a=\log k$, $t=t_{j,\,m-1}$ and $s=t_{j,\,m}$ we obtain $$0\leq \frac{1}{k^{\exp(-1/t_{j,\,m-1})}}-\frac{1}{k^{\exp(-1/t_{j,\,m})}}\leq \frac{\log k}{k^{\exp(-1/t_{j,\,m-1})}}\eee^{-1/t_{j,\,m}}\Big(\frac{1}{t_{j,\,m-1}}-\frac{1}{t_{j,\,m}}\Big).$$ In view of the inequality $x\eee^{-ax}\leq (\eee a)^{-1}$ for $x\geq 0$ we conclude that the right-hand side does not exceed
\begin{multline*}
\frac{1}{\eee}\eee^{1/t_{j,\,m-1}-1/t_{j,\,m}}\Big(\frac{1}{t_{j,\,m-1}}-\frac{1}{t_{j,\,m}}\Big)\leq \frac{1}{\eee} \exp\Big(\frac{2^{-j}(v_n-v_{n+1})}{v_{n+1}^2}\Big)\frac{2^{-j}(v_n-v_{n+1})}{v_{n+1}^2}\\\leq \frac{2^{-j}(v_n-v_{n+1})}{v_{n+1}^2},
\end{multline*}
where the last inequality is justified by our present choice of $j$. Thus, $$B(u,k,n)\leq \frac{|u|2^{-j}\log n}{(\log M(s_{n+1}))^{1/2}}\frac{v_n-v_{n+1}}{v_{n+1}^2}=:C_j(u,n)$$ whenever $k\geq M(s_{n+1})+1$.

Assume now that $2^j\leq v_{n+1}^{-2}(v_n-v_{n+1})$ for nonnegative integer $j$. Using a trivial estimate $$\frac{1}{k^{\exp(-1/t_{j,\,m-1})}}-\frac{1}{k^{\exp(-1/t_{j,\,m})}}\leq 1$$ we infer
$$B(u,k,n)\leq \frac{|u|\log n}{(\log M(s_{n+1}))^{1/2}}=:C_j(u,n)$$ whenever $k\geq M(s_{n+1})+1$.

Using now $1+x\le \eee^x$ for $x\in\mathbb{R}$ we arrive at
\begin{equation*}
\me\big[\exp\big(\pm u(Y_n(t_{j,\,m})-Y_n(t_{j,\,m-1}))\big)\big]\leq \exp\Big(\frac{u^2}{2}\frac{2^{-j}(v_n-v_{n+1})}{v_{n+1}^2}\eee^{C_j(u,n)}\Big),\quad u\in\mathbb{R}.
\end{equation*}
By Markov's inequality and the inequality $\eee^{u|x|}\le \eee^{ux}+\eee^{-ux}$ for $x\in\mathbb{R}$,
\begin{multline*}
\mathbb{P}\big\{f(s_n)|Y_n(t_{j,\,m})-Y_n(t_{j,\,m-1})|>\varepsilon a_j \big\}\leq \exp\Big(-u\frac{\varepsilon a_j}{f(s_n)}\Big)\mathbb{E}[\exp(u|Y_n(t_{j,\,m})-Y_n(t_{j,\,m-1})|)]\\\leq  2\exp\Big(-\frac{u\varepsilon a_j}{f(s_n)}+\frac{u^2}{2}\frac{2^{-j}(v_n-v_{n+1})}{v_{n+1}^2}\eee^{C_j(u,n)}\Big).
\end{multline*}
Put $$u=\frac{\varepsilon 2^{j/2}}{f(s_n)}\frac{v_{n+1}^2}{v_n-v_{n+1}}\quad\text{and}\quad k_n:=\frac{1}{(f(s_n))^2}\frac{v_{n+1}^2}{v_n-v_{n+1}}.$$ In view of $$(\log M(s_{n+1}))^{-1/2}~\sim~ n^{\gamma/2-1/2},~\Big(\frac{v_{n+1}}{v_n-v_{n+1}}\Big)^{1/2}~\sim~(1-\gamma)^{-1/2} n^{\gamma/2},\quad n\to\infty$$ and $$\lim_{n\to\infty} (\log 1/s_n)v_{n+1}=1$$ we conclude that, for $j$ satisfying $2^{j}\geq v_{n+1}^{-2}(v_n-v_{n+1})$,
\begin{multline*}
C_j(u,n)=\frac{|u|2^{-j}\log n}{(\log M(s_{n+1}))^{1/2}}\frac{v_n-v_{n+1}}{v_{n+1}^2}=\frac{\varepsilon 2^{-j/2}\log n}{f(s_n)(\log M(s_{n+1}))^{1/2}}\\\leq \frac{\varepsilon \log n}{f(s_n)(\log M(s_{n+1}))^{1/2}}\frac{v_{n+1}}{(v_n-v_{n+1})^{1/2}}\\=\frac{\varepsilon 2^{1/2}((\log 1/s_n) v_{n+1})^{1/2}(\log^{(3)}1/s_n)^{1/2}\log n}{(\log M(s_{n+1}))^{1/2}}\Big(\frac{v_{n+1}}{v_n-v_{n+1}}\Big)^{1/2}\\~\sim~ \varepsilon 2^{1/2}n^{\gamma-1/2}(\log n)^{3/2}~\to~0,\quad n\to\infty.
\end{multline*}
For nonnegative integer $j$ satisfying $2^{j}\leq v_{n+1}^{-2}(v_n-v_{n+1})$, $C_j(u,n)$ admits the same upper bound:
\begin{multline*}
C_j(u,n)=\frac{|u|\log n}{(\log M(s_{n+1}))^{1/2}}=\frac{\varepsilon 2^{j/2}\log n}{f(s_n)(\log M(s_{n+1}))^{1/2}}\frac{v_{n+1}^2}{v_n-v_{n+1}}\\\leq \frac{\varepsilon \log n}{f(s_n)(\log M(s_{n+1}))^{1/2}}\frac{v_{n+1}}{(v_n-v_{n+1})^{1/2}}~\to~0,\quad n\to\infty.
\end{multline*}
Therefore, for sufficiently large $n$ such that $k_n>\varepsilon^{-2} \log 2$ and $\eee^{C_j(u,n)}\leq 3/2$,
\begin{multline*}
\sum_{j\ge 0}\mathbb{P}\Big\{\max_{1\le m \le 2^j}f(s_n)|Y_n(t_{j,\,m})-Y_n(t_{j,\,m-1})|>\varepsilon a_j \Big\}\\ \le \sum_{j\ge 0}2^j 2\exp(-\varepsilon^2(j+1)k_n)\exp\big(3\varepsilon^2k_n/4\big)=\frac{2\exp(-\varepsilon^2 k_n/4)}{1-2\exp(-\varepsilon^2 k_n)}.
\end{multline*}
Since $k_n\sim 2 n^\gamma\log n$ as $n\to\infty$, \eqref{5.12} follows.

Now we comment on the reason behind the passage from $Y^\ast_n$ to $Y_n$ via a logarithmic transformation. In order to have in the present context a successful approximation with the help of a dyadic partition of an interval $[\lambda_{n+1},\lambda_n]$, say, the endpoints should satisfy $\lim_{n\to\infty}(\lambda_n/\lambda_{n+1})=1$. This is the case for $\lambda_n=v_n$ and not the case for $\lambda_n=s_n$.
\end{proof}

We are prepared to prove Proposition \ref{pr1}.
\begin{proof}[Proof of Proposition \ref{pr1}]
We only prove \eqref{5.1} because \eqref{5.2} follows from \eqref{5.1} by replacing $\eta_k$ with $-\eta_k$. Recall our convention that $\sigma^2=1$.

Choose any $\gamma>0$ sufficiently close to $0$, put $s_n=\exp(-\exp(n^{1-\gamma}))$ for $n\in\mn$ and select $\rho=\rho(\gamma)$ such that \eqref{5.7} holds true. Let $M(s)=\lfloor \log 1/s/\log\log 1/s\rfloor$ for $s\in (0,\eee^{-\eee})$. 
By Lemmas \ref{lemma_3} and \ref{lemma4} and the fact that $\mathbb{E}\big(\eta_k\1_{(A_{k,\rho}(s))^c}\big)=-\mathbb{E}\big(\eta_k\1_{(A_{k,\rho}(s))^c}\big)$,
\begin{equation}\label{5.16}
\limsup_{n\to\infty} f(s_n)\sum_{k\ge M(s_n)+1}\frac{(\log k)^{-1/2}}{k^{1/2+s_n}}\eta_k\le 1+\gamma\quad\text{\rm a.s.}
\end{equation}
Lemma \ref{lemma5} in combination with relation \eqref{5.16} entails
$$\limsup_{s\to 0+} f(s)\sum_{k\ge M(s)+1}\frac{(\log k)^{-1/2}}{k^{1/2+s}}\eta_k\le 1+\gamma\quad\text{\rm a.s.}$$ Finally, by Lemma \ref{lemma_2},
\begin{equation*}
\limsup_{s\to 0+}f(s)\sum_{k\ge 2}\frac{(\log k)^{-1/2}}{k^{1/2+s}}\eta_k\le 1+\gamma\quad\text{\rm a.s.,}
\end{equation*}
which entails \eqref{5.1} because the left-hand side does not depend on $\gamma$.
\end{proof}

\begin{proposition}\label{pr2}
Under the assumptions of Theorem \ref{main},
\begin{equation}\label{5.17}
\limsup_{s\to 0+}\Big(\frac{1}{\log 1/s\,\log^{(3)}1/s}\Big)^{1/2}\sum_{k\ge 2}\frac{(\log k)^{-1/2}}{k^{1/2+s}}\eta_k\geq(2\sigma^2)^{1/2}\quad\text{\rm a.s.}
\end{equation}
and
\begin{equation}\label{5.18}
\liminf_{s\to 0+}\Big(\frac{1}{\log 1/s \,\log^{(3)}1/s 
}\Big)^{1/2}\sum_{k\ge 2}\frac{(\log k)^{-1/2}}{k^{1/2+s}}\eta_k\leq-(2\sigma^2)^{1/2}\quad\text{\rm a.s.}
\end{equation}
\end{proposition}
Proposition \ref{pr2} will be proved by using Lemmas \ref{lemma_2} and \ref{lemma_3}, together with two additional lemmas. In the first of these, we show that an initial and a final fragments of the series in question do not contribute to the LIL.

As before, we assume without further notice that the assumptions of Theorem \ref{main} hold true and that $\sigma^2=1$. When proving Proposition \ref{pr2} we use the sets $A_{k,\rho}(s)$ and the corresponding variables $\widetilde{\eta}_{k,\rho}(s)$ with $\rho=1$.

\begin{lemma}\label{lemma5.7}
Fix any $\gamma>0$ and put $\mathfrak{s}_n:=\exp(-\exp(n^{1+\gamma}))$ for $n\ge 1 
$. Let $N_1$ and $N_2$ be functions which take positive integer values, may depend on $\gamma$ and satisfy
$N_1(s)\geq 2$ for small positive $s$, $\lim_{s\to 0+}(\log\log N_1(s)/\log 1/s)=0$, $\lim_{s\to 0+}(\log\log N_2(s)/\log 1/s)=1$ and $\lim_{s\to 0+}s\log N_2(s)=0$. Then
\begin{equation}\label{5.19}
\lim_{n\to\infty}f(\mathfrak{s}_n) \sum_{k=2}^{N_1(\mathfrak{s}_n)}\frac{(\log k)^{-1/2}}{k^{1/2+\mathfrak{s}_n}}\widetilde{\eta}_{k,1}(\mathfrak{s}_n)=0\quad\text{\rm a.s.}
\end{equation}
and
\begin{equation}\label{5.20}
\lim_{n\to\infty}f(\mathfrak{s}_n) \sum_{k\ge N_2(\mathfrak{s}_n)+1}\frac{(\log k)^{-1/2}}{k^{1/2+\mathfrak{s}_n}}\widetilde{\eta}_{k,1}(\mathfrak{s}_n)=0\quad\text{\rm a.s.}
\end{equation}
\end{lemma}
\begin{proof}
As in the proof of Lemma \ref{lemma4}, we obtain the result, with $f^\ast$ replacing $f$. For $s>0$ close to $0$, put
\begin{equation*}
Z_1(s):=f^\ast(s)
\sum_{k=2}^{N_1(s)}\frac{(\log k)^{-1/2}}{k^{1/2+s}}\widetilde{\eta}_{k,1}(s).
\end{equation*}
The reasoning used to derive both \eqref{5.19} and \eqref{5.20} is similar to the one applied in the proof of Lemma \ref{lemma4}. Therefore, we give a proof of \eqref{5.19} and indicate the only minor change needed for a proof of
\eqref{5.20}.

Regarding \eqref{5.19}, in view of the direct part of the Borel–Cantelli lemma, it is sufficient to demonstrate that, for all $\varepsilon>0$,
\begin{equation}\label{5.21}
\sum_{n\ge 1}\mathbb{P}\{Z_1(\mathfrak{s}_n)>\varepsilon\}<\infty.
\end{equation}
To this end, we obtain a counterpart of \eqref{5.9}, for $u\in\mathbb{R}$,
\begin{equation*}
\mathbb{E}[\eee^{uZ_1(s)}]\le\exp\Big(\frac{u^2(f^\ast(s))^2}{2}\sum_{k=2}^{N_1(s)}\frac{(\log k)^{-1}}{k^{1+2s}}\exp\Big(\frac{\sqrt{2}|u|}{ \log^{(3)}1/s}\Big)\Big).
\end{equation*}
For each fixed $s>0$, the function $x\to(\log x)^{-1}x^{-1-2s}$ is decreasing on $(1, \infty)$. The assumption $\lim_{s\to 0+}(\log\log N_1(s)/\log 1/s)=0$ entails $\lim_{s\to 0+}s\log N_1(s)=0$. Hence,
\begin{multline}\label{5.22}
\sum_{k=2}^{N_1(s)}\frac{(\log k)^{-1}}{k^{1+2s}}\le \frac{(\log 2)^{-1}}{2^{1+2s}}+\int_2^{N_1(s)}\frac{(\log x)^{-1}}{x^{1+2s}}{\rm d}x\\=O(1)+\int_{(2\log 2) s}^{2s\log N_1(s)}
y^{-1}\eee^{-y}{\rm d}y=O(\log\log N_1(s)),\quad s\to 0+.
\end{multline}
Our choice of $N_1$ guarantees that there exists an $r>0$ close to $0$ and such that
\begin{equation*}
\sum_{k=2}^{N_1(s)}\frac{(\log k)^{-1}}{k^{1+2s}}\le r g(s)
\end{equation*}
for small positive $s$, and also $(1-\delta)(1+\gamma)>1$, where $\delta:=r(4\varepsilon^2)^{-1}\exp(\sqrt{2}\varepsilon^{-1})$ with $\varepsilon$ appearing in \eqref{5.21}. Then,
for $u\in\mathbb{R}$ and small $s>0$,
\begin{multline*}
\mathbb{E}[\eee^{uZ_1(s)}]\le\exp\Big(\frac{ru^2 g(s)(f^\ast(s))^2}{2}\exp\Big(\frac{\sqrt{2}|u|}{ \log^{(3)}1/s}\Big)\Big)=\exp\Big(\frac{ru^2}{4 \log^{(3)}1/s}\exp\Big(\frac{\sqrt{2}|u|}{ \log^{(3)}1/s}\Big)\Big).
\end{multline*}
Applying Markov's inequality with $u=(1/\varepsilon)\log^{(3)} 1/\mathfrak{s}_n$ yields, for large $n$,
\begin{multline*}
\mathbb{P}\{Z_1(\mathfrak{s}_n)>\varepsilon\}\le \eee^{-u\varepsilon}\mathbb{E}[\eee^{uZ_1(\mathfrak{s}_n)}]\le\exp\big(-\big(1-r(4\varepsilon^2)^{-1}\eee^{\sqrt{2}\varepsilon^{-1}}\big)\log^{(3)} 1/\mathfrak{s}_n\big)=\frac{1}{n^{(1-\delta)(1+\gamma)}},
\end{multline*}
thereby proving \eqref{5.21} and hence \eqref{5.19}.

The proof of \eqref{5.20} is analogous to that of \eqref{5.19}. The corresponding version of \eqref{5.22} reads
\begin{equation}\label{5.23}
\sum_{k\ge N_2(s)+1}\frac{(\log k)^{-1}}{k^{1+2s}}\le \int_{N_2(s)}^{\infty} \frac{(\log x)^{-1}}{x^{1+2s}}{\rm d}x=\int_{2s\log N_2(s)}^\infty y^{-1}\eee^{-y}{\rm d}y=o(1)
\end{equation}
as $s\to 0+$.
\end{proof}

Lemma \ref{lemma5.8} treats the component of the series in question which gives a principal contribution to the LIL. Our proof is based on the converse part of the Borel–Cantelli lemma, which requires independence. The independence requirement complicates to some extent a selection of the essential component.
\begin{lemma}\label{lemma5.8}
Fix sufficiently small $\delta>0$, pick $\gamma>0$ satisfying $(1+\gamma)(1-\delta^2/8)<1$ and let, as before, $\mathfrak{s}_n=\exp(-\exp(n^{1+\gamma}))$ for $n\in\mn$. Then
\begin{equation*}
\limsup_{n\to\infty} f(\mathfrak{s}_n) \sum_{k\ge 2}\frac{(\log k)^{-1/2}}{k^{1/2+\mathfrak{s}_n}}\widetilde{\eta}_{k,1}(\mathfrak{s}_n)\ge 1-\delta\quad\text{\rm a.s.}
\end{equation*}
\end{lemma}
\begin{proof}
Observe that $\lim_{n\to\infty}(\mathfrak{s}_{n+1}/\mathfrak{s}_n)=\infty$. Functions $s\mapsto \log\log N_1(s)/\log 1/s$ in Lemma \ref{lemma5.7} can tend to $0$ as fast as we please. Thus, we can choose $N_1$ satisfying
$$\frac{\log\log N_1(\mathfrak{s}_{n+1})}{\log 1/\mathfrak{s}_n}=\frac{\log\log N_1(\mathfrak{s}_{n+1})}{\log 1/\mathfrak{s}_{n+1}}\frac{\log 1/\mathfrak{s}_{n+1}}{\log 1/\mathfrak{s}_n}~\to~+\infty,\quad n\to\infty.$$ Let $N_2$ be any function satisfying the assumptions of Lemma \ref{lemma5.7}. Since $$\lim_{n\to\infty}\frac{\log\log N_2(\mathfrak{s}_n)}{\log 1/\mathfrak{s}_n}=1,$$ we conclude that there exists $n_0\in\mn$ such that $$\frac{\log\log N_1(\mathfrak{s}_{n+1})}{\log 1/\mathfrak{s}_n}\geq \frac{\log\log N_2(\mathfrak{s}_n)}{\log 1/\mathfrak{s}_n}\quad\text{for all}~n\geq n_0,$$ whence
\begin{equation}\label{5.24}
N_1(\mathfrak{s}_{n+1})\ge N_2(\mathfrak{s}_n),\quad n\ge n_0.
\end{equation}

In view of Lemma \ref{lemma5.7}, it is sufficient to check that
        \begin{equation}\label{5.25}
            \limsup_{n\to\infty}Z_2(\mathfrak{s}_n)\ge 1-\delta\quad\text{a.s.},
        \end{equation}
        where, for small $s>0$,
        \begin{equation*}
            Z_2(s):=f(s)
\sum_{k=N_1(s)+1}^{N_2(s)}\frac{(\log k)^{-1/2}}{k^{1/2+s}}\widetilde{\eta}_{k,1}(s).
        \end{equation*}
We shall prove that there exists $\overline{s}>0$ such that for all $s\in(0,\overline{s})$,
\begin{equation}\label{5.26}
\mathbb{P}\{Z_2(s)>1-\delta\}\ge 3^{-1}\eee^{-(1-\delta^2/8)\log^{(3)}1/s}.
\end{equation}
As a consequence,
\begin{equation*}
\sum_{n\ge n_1}\mathbb{P}\{Z_2(\mathfrak{s}_n)>1-\delta\}\ge 3^{-1}\sum_{n\ge n_1}\frac{1}{n^{(1+\gamma)(1-\delta^2/8)}}=\infty,
\end{equation*}
where $n_1\ge n_0$ is chosen to ensure that $\mathfrak{s}_n<\overline{s}$ for $n\ge n_1$. In view of \eqref{5.24} the random variables
$Z_2(\mathfrak{s}_{n_1})$, $Z_2(\mathfrak{s}_{n_1+1}),\ldots$ are independent. Hence, divergence of the series entails \eqref{5.25} by the converse part of the Borel-Cantelli lemma.

When proving \eqref{5.26} we shall use the event
\begin{equation*}
U_s:=\big\{1-\delta<Z_2(s)\le 1\big\}=\big\{(1-\delta)V(s)< W(s)/(g(s))^{1/2}\le V(s)\big\},
\end{equation*}
where $V(s)=(2\,\log^{(3)}1/s)^{1/2}$ and
\begin{equation*}
W(s):=\frac{Z_2(s)}{f(s)}=\sum_{k=N_1(s)+1}^{N_2(s)}\frac{(\log k)^{-1/2}}{k^{1/2+s}}\widetilde{\eta}_{k,1}(s).
\end{equation*}
For $u\in\mathbb{R}$ and small $s>0$, let $\mathbb{Q}_{s,u}$ be a probability measure on $(\Omega, \mathfrak{F})$ defined by
\begin{equation}
\mathbb{Q}_{s,u}(A)=\frac{\mathbb{E}\big(\eee^{uW(s)/(g(s))^{1/2}}\1_A\big)}{\mathbb{E}\eee^{uW(s)/(g(s))^{1/2}}}, \quad A\in\mathfrak{F}.
\end{equation}
Then
\begin{multline}\label{5.27}
\Big(\mathbb{E}\eee^{u(W(s)/(g(s))^{1/2}-V(s))}\Big)\mathbb{Q}_{s,u}(U_s)=\eee^{-uV(s)}\me\eee^{uW(s)/(g(s))^{1/2}}\1_{U_s}\\\leq \mathbb{P}(U_s)\le\mathbb{P}\{Z_2(s)>1-\delta\}.
\end{multline}
We shall demonstrate that by selecting an appropriate $u=u(s)=O((\log^{(3)}1/s)^{1/2})$, the expectation on the left-hand side is bounded from below by $\eee^{-(1-\delta^2/8)\log^{(3)}1/s}$ for small positive $s$. Also, we shall prove that $\mathbb{Q}_{s,u}(U_s)\ge 1/3$, thereby deriving \eqref{5.26}.

We start by showing that, with $u=O((\log^{(3)}1/s)^{1/2})$,
\begin{equation}\label{5.28}
\mathbb{E}[\eee^{uW(s)/(g(s))^{1/2}}]~\sim~ \eee^{u^2/2+u^2 h(s)},\quad s\to 0+
\end{equation}
for some function $h$ satisfying $\lim_{s\to 0+}h(s)=0$. Put
\begin{equation*}
\xi_k(s)=\frac{(\log k)^{-1/2}}{(g(s))^{1/2}k^{1/2+s}}\widetilde{\eta}_{k,1}(s),\quad k\ge 2, \quad s\in(0, \eee^{-\eee}).
\end{equation*}
As a consequence, for $u\in\mathbb{R}$,
\begin{equation}\label{5.29}
uW(s)/(g(s))^{1/2}=\sum_{k=N_1(s)+1}^{N_2(s)}u\xi_k(s).
\end{equation}
According to the second inequality in \eqref{5.8},
\begin{equation*}
|u\xi_k(s)|=O(1/\log\log 1/s)=o(1),\quad s\to 0+\quad\text{\rm a.s.}
\end{equation*}
for each $k\geq 2$. Using
\begin{equation*}
\eee^x=1+x+x^2/2+o(x^2)\quad\text{and}\quad\log(1+x)=x+O(x^2),\quad x\to 0,
\end{equation*}
we obtain
\begin{multline}\label{5.30}
\mathbb{E}[\eee^{uW(s)/(g(s))^{1/2}}]=\prod_{k=N_1(s)+1}^{N_2(s)}\mathbb{E}[\exp(u\xi_k(s))]=\prod_{k=N_1(s)+1}^{N_2(s)}\mathbb{E}\big[1+u\xi_k(s)+u^2\xi_k^2(s)(1/2+o(1))\big]\\=
\exp\sum_{k=N_1(s)+1}^{N_2(s)}\log\big(1+u^2\mathbb{E}\big[\xi_k^2(s)(1/2+o(1))\big]\big)\\=\exp\Big(u^2\big(1/2+o(1)\big)\sum_{k=N_1(s)+1}^{N_2(s)}\mathbb{E}\xi_k^2(s)+u^4O\Big(\sum_{k=N_1(s)+1}^{N_2(s)}(\mathbb{E}\xi_k^2(s))^2\Big)\Big).
\end{multline}
Here, we have used the fact that the $o(1)$ term can be chosen nonrandom. In view of \eqref{5.22} and \eqref{5.23}
\begin{equation*}
\sum_{k=N_1(s)+1}^{N_2(s)}\frac{(\log k)^{-1}}{k^{1+2s}}~\sim~ g(s),\quad s\to 0+.
\end{equation*}
The latter, combined with uniformity in the integer $k\in[N_1(s)+1,N_2(s)]$ of the limit relation
\begin{equation*}
\mathbb{E}[\widetilde{\eta}_{k,1}^2(s)]=\mathbb{E}[\eta_k^2\1_{(A_{k,1}(s))^c}]-\big(\mathbb{E}[\eta_k\1_{(A_{k,1}(s))^c}]\big)^2~\to~ 1,\quad s\to 0+,
\end{equation*}
results in
\begin{equation}\label{5.31}
\sum_{k=N_1(s)+1}^{N_2(s)}\mathbb{E}[\xi_k^2(s)]=\frac{1}{g(s)}\sum_{k=N_1(s)+1}^{N_2(s)}\frac{(\log k)^{-1}}{k^{1+2s}}\mathbb{E}\widetilde{\eta}_{k,1}^2(s)\to 1,\quad s\to 0+.
\end{equation}
Finally,
\begin{multline}\label{5.32}
u^2\sum_{k=N_1(s)+1}^{N_2(s)}(\mathbb{E}[\xi_k^2(s)])^2=\frac{u^2}{(g(s))^2}\sum_{k=N_1(s)+1}^{N_2(s)}\frac{(\log k)^{-2}}{k^{2+4s}}(\mathbb{E}[\widetilde{\eta}_{k,1}^2(s)])^2\\
\le \frac{u^2}{(g(s))^2}\sum_{k\ge N_1(s)+1}\frac{(\log k)^{-2}}{k^2}=o(1),\quad s\to 0+
\end{multline}
because both factors converge to $0$ as $s\to 0+$. Now relations \eqref{5.30}, \eqref{5.31} and \eqref{5.32} entail \eqref{5.28}.

Observe that formula \eqref{5.28}, with $u\in\mathbb{R}$ fixed, reads $\lim_{s\to 0+} \mathbb{E}[\eee^{uW(s)/(g(s))^{1/2}}]=\eee^{u^2/2}$ as $s\to0+$, which implies a central limit theorem $W(s)/(g(s))^{1/2}\dod \mathcal{N}(0,1)$ as $s\to0+$. Here, $\dod$ denotes convergence in distribution and $\mathcal{N}(0,1)$ denotes a random variable with the normal distribution with mean $0$ and variance $1$.

Passing to the proof of \eqref{5.26}, put
\begin{equation*}
u=u(s)=\sqrt{2}(1-\delta/2)(\log^{(3)}1/s)^{1/2}.
\end{equation*}
Formula \eqref{5.28} entails
\begin{equation}\label{5.33}
\mathbb{E}[\eee^{u(W(s)/(g(s))^{1/2}-V(s))}]=\eee^{-(1-\delta^2/4)\log^{(3)}1/s+o(\log^{(3)}1/s)}\ge\eee^{-(1-\delta^2/8)\log^{(3)}1/s}
\end{equation}
for small $s>0$. Now we show that the $\mathbb{Q}_{s,u}$-distribution of $W(s)/(g(s))^{1/2}-(u+2uh(s))$ converges weakly as $s\to0+$ to the $\mathbb{P}$-distribution of $\mathcal{N}(0,1)$.
To this end, we prove convergence of the moment generating functions. Let $\mathbb{E}_{\mathbb{Q}_{s,u}}$ denote the expectation with respect to the probability measure $\mathbb{Q}_{s,u}$. Invoking \eqref{5.28} we conclude that, for $t\in\mathbb{R}$,
\begin{multline*}
\mathbb{E}_{\mathbb{Q}_{s,u}}[\eee^{t(W(s)/(g(s))^{1/2}-(u+2uh(s)))}]~\sim~\frac{\mathbb{E}[\eee^{(t+u)W(s)/(g(s))^{1/2}}]}{\mathbb{E}[\eee^{uW(s)/(g(s))^{1/2}}]}\eee^{-t(u+2uh(s))}\\
=\exp\big((t+u)^2/2+(t+u)^2h(s)-u^2/2-u^2h(s)-t(u+2uh(s))\big)\\=\exp\big((1/2+h(s))t^2\big)\to\exp(t^2/2)=\mathbb{E}[\exp(t\,\mathcal{N}(0,1))],\quad s\to 0+.
\end{multline*}
As a consequence of the weak convergence, we infer
\begin{multline*}
\limsup_{s\to0+}\mathbb{Q}_{s,u}\big\{W(s)/(g(s))^{1/2}\le(1-\delta)V(s)\big\}\\ \le \lim_{s\to0+}\mathbb{Q}_{s,u}\big\{W(s)/(g(s))^{1/2}\le u+2uh(s)\big\}=\mathbb{P}\{\mathcal{N}(0,1)
\le 0\}=1/2.
\end{multline*}
Further, the relation $\lim_{s\to0+}(V(s)-(u+2uh(s)))=+\infty$ entails
\begin{equation*}
\lim_{s\to0+}\mathbb{Q}_{s,u}\big\{W(s)/(g(s))^{1/2}\le V(s)\big\}=1
\end{equation*}
and thereupon
\begin{multline}\label{5.34}
\mathbb{Q}_{s,u}(U_s)=\mathbb{Q}_{s,u}\big\{W(s)/(g(s))^{1/2}\le V(s)\big\}-\mathbb{Q}_{s,u}\big\{W(s)/(g(s))^{1/2}\le (1-\delta)V(s)\big\}\ge 1/3
\end{multline}
for small $s>0$. Now \eqref{5.26} follows from \eqref{5.27}, \eqref{5.33} and \eqref{5.34}. The proof of Lemma \ref{lemma5.8} is complete.
        \end{proof}
\begin{proof}[Proof of Proposition \ref{pr2}]
It suffices to prove \eqref{5.17}. To this end, pick sufficiently small $\delta>0$ and $\gamma>0$, and let $\mathfrak{s}_n=\exp(-\exp(n^{1+\gamma}))$ for $n\geq 1$. We shall invoke a representation
\begin{multline*}
f(\mathfrak{s}_n) \sum_{k\ge 2}\frac{(\log k)^{-1/2}}{k^{1/2+\mathfrak{s}_n}}\eta_k=f(\mathfrak{s}_n) \sum_{k=2}^{\lfloor(\log 1/\mathfrak{s}_n)^{1/2}\rfloor}\frac{(\log k)^{-1/2}}{k^{1/2+\mathfrak{s}_n}}\eta_k\\-f(\mathfrak{s}_n) \sum_{k=2}^{\lfloor(\log 1/\mathfrak{s}_n)^{1/2}\rfloor}\frac{(\log k)^{-1/2}}{k^{1/2+\mathfrak{s}_n}}(\eta_k\1_{(A_{k,1}(\mathfrak{s}_n))^c}-\mathbb{E}[\eta_k\1_{(A_{k,1}(\mathfrak{s}_n))^c}])\\+f(\mathfrak{s}_n) \sum_{k\ge\lfloor(\log 1/\mathfrak{s}_n)^{1/2}\rfloor+1}\frac{(\log k)^{-1/2}}{k^{1/2+\mathfrak{s}_n}}(\eta_k\1_{A_{k,1}(\mathfrak{s}_n)}-\mathbb{E}[\eta_k\1_{A_{k,1}(\mathfrak{s}_n)}])\\+f(\mathfrak{s}_n) \sum_{k\ge2}\frac{(\log k)^{-1/2}}{k^{1/2+\mathfrak{s}_n}}(\eta_k\1_{(A_{k,1}(\mathfrak{s}_n))^c}-\mathbb{E}[\eta_k\1_{(A_{k,1}(\mathfrak{s}_n))^c}]).
        \end{multline*}
Here, we have used that $\mathbb{E}[\eta]=0$. The first three terms on the right-hand side converge to $0$ a.s. as $n\to\infty$ by Lemma \ref{lemma_2}, formula \eqref{5.19} of Lemma \ref{lemma5.7}, with $N_1(s)=M(s)=\lfloor(\log 1/s)^{1/2}\rfloor$, and Lemma \ref{lemma_3}, respectively. Lemma \ref{lemma5.8} ensures that
\begin{equation*}
\limsup_{s\to0+}f(s) \sum_{k\ge 2}\frac{(\log k)^{-1/2}}{k^{1/2+s}}\eta_k\ge\limsup_{n\to\infty}f(\mathfrak{s}_n) \sum_{k\ge 2}\frac{(\log k)^{-1/2}}{k^{1/2+\mathfrak{s}_n}}\eta_k\ge 1-\delta\quad\text{\rm a.s.}
\end{equation*}
Sending $\delta$ to $0+$ we arrive at \eqref{5.17}.
\end{proof}
\begin{proof}[Proof of Theorem \ref{main}]
A combination of Propositions \ref{pr1} and \ref{pr2} yields \eqref{3.1} and \eqref{3.2}. To prove \eqref{3.3} we put $H:=\{z\in\mathbb{C}: {\rm Re}\,z>0\}$ and $$X(z):=\sum_{k\ge2}\frac{(\log k)^{-1/2}}{k^{1/2+z}}\eta_k,\quad z\in H,$$ and note that the so defined $X$ is a random analytic function, see p.~247 in \cite{Buraczewski etal:2023}. Hence, its restriction to positive arguments $s\to X(s)=\sum_{k\ge2}(\log k)^{-1/2}k^{-1/2-s}\eta_k$ is a.s. continuous, and so is $s\mapsto (2\sigma^2\log 1/s\, \log^{(3)}1/s )^{-1/2}X(s)$ on $(0,\eee^{-\eee})$. In view of \eqref{3.1} and \eqref{3.2}, we obtain \eqref{3.3} with the help of the intermediate value theorem for continuous functions.
\end{proof}

\section{Proof of Theorem \ref{thm:flt}}\label{sect:flt}

It is more convenient to prove the result in an equivalent form
$$\Big(\frac{1}{s^{1/2}}\sum_{k\geq 2}\frac{(\log k)^{-1/2}}{k^{1/2+\exp(-ts)}}\eta_k\Big)_{t\geq 0}~\Longrightarrow~ (\sigma B(t))_{t\geq 0},\quad s\to+\infty$$ on $C[0,\infty)$. Without loss of generality we can and do assume that $\sigma^2=1$.

In what follows we write $X$ for $X_{-1/2}$. We use a standard approach, which consists of two steps: (a) proving weak convergence of finite-dimensional distributions; (b) checking tightness.

\noindent (a) If $t=0$, then, for $s>0$, $X(\eee^{-ts})=X(1)=\sum_{k\geq 2}(\log k)^{-1/2} k^{-3/2}\eta_k$, and $s^{-1/2}X(1)$ converges in probability to $B(0)=0$ as $s\to +\infty$.

Thus, it suffices to show that, for $t_1, t_2\in (0,\infty)$ (we do not need to consider $t=0$),
\begin{equation}\label{eq:covar}
\me \big[X(\eee^{-t_1s})X(\eee^{-t_2s})\big] ~\sim~ \min(t_1, t_2)s,\quad s\to +\infty
\end{equation}
and check the Lindeberg-Feller condition: for all $\varepsilon>0$ and each fixed $t>0$,
\begin{equation}\label{eq:Lindeberg}
\lim_{s\to +\infty}\frac{1}{s}\sum_{k\geq 2}\me\Big[\Big(\frac{(\log k)^{-1/2}}{k^{1/2+\exp(-ts)}}\eta_k\Big)^2\1_{\{(\log k)^{-1/2}|\eta_k|>\varepsilon k^{1/2+\exp(-ts)}s^{1/2}\}}\Big]=0.
\end{equation}
\noindent {\sc Proof of \eqref{eq:covar}.} Using monotonicity to pass from the series to an integral we obtain
\begin{multline*}
\me \Big[\sum_{k\geq 2}\frac{(\log k)^{-1/2}}{k^{1/2+\exp(-t_1s)}}\eta_k\sum_{j\geq 2}\frac{(\log j)^{-1/2}}{j^{1/2+\exp(-t_2s)}}\eta_j\Big]=\sum_{k\geq 2}\frac{(\log k)^{-1}}{k^{1+\exp(-t_1s)+\exp(-t_2s)}}\\~\sim~ \int_\eee^\infty\frac{(\log y)^{-1}}{y^{1+\exp(-t_1s)+\exp(-t_2s)}}{\rm d}y=\int_1^\infty \frac{\eee^{-(\exp(-t_1s)+\exp(-t_2s))y}}{y}{\rm d}y~\sim~ -\log (\eee^{-t_1s}+\eee^{-t_2s})\\~\sim~ \min(t_1,t_2)s,\quad s\to +\infty.
\end{multline*}
{\sc Proof of \eqref{eq:Lindeberg}.} For each $k\geq 2$ and each $s>0$, $(\log k)^{-1/2}k^{-1/2-\exp(-ts)} \leq (\log 2)^{-1/2}=:1/A$. Hence, the expression under the limit on the left-hand side of \eqref{eq:Lindeberg} does not exceed $$\frac{1}{s}\sum_{k\geq 2}\frac{(\log k)^{-1}}{k^{1+2\exp(-ts)}}\me [\eta^2\1_{\{|\eta|>As^{1/2}\}}].$$  As shown in the proof of \eqref{eq:covar}, $\sum_{k\geq 2} (\log k)^{-1} k^{-1-2\exp(-ts)} \sim ts$ as $s\to +\infty$. Further, $\me [\eta^2]<\infty$ entails $\lim_{s\to +\infty}\me [\eta^2\1_{\{|\eta|>A s^{1/2}\}}]=0$. With these at hand, \eqref{eq:Lindeberg} follows.

\noindent (b) We have to prove tightness on $C[0,T]$ (the set of continuous functions defined on $[0,T]$) for each $T>0$. Since $(B(t))_{t\in [0,T]}$ has the same distribution as $T^{1/2}(B(t))_{t\in [0,1]}$, it is enough to investigate the case $T=1$ only.

Let $M^\ast: (0,\infty)\to \mn_0$ denote a function satisfying $\lim_{s\to +\infty}M^\ast(s)=+\infty$ and $M^\ast(s)=o(s)$ as $s\to+\infty$. We shall need a relation $\sup_{k\leq n}\,|T_k|=O(n^{1/2})$ in probability as $n\to\infty$, which is a consequence of $n^{-1/2}\max_{k\leq n}\,|T_k|\dod |{\rm Normal}\,(0,1)|$ as $n\to\infty$. Repeating the proof of Lemma \ref{lemma_2}, with the aforementioned limit relation replacing \eqref{5.3}, we infer
\begin{equation}\label{eq:inter3}
\lim_{s\to +\infty} \frac{1}{s^{1/2}}\sup_{t\in [0,1]}\,\Big|\sum_{k=2}^{M^\ast(s)}\frac{(\log k)^{-1/2}}{k^{1/2+\exp(-ts)}}\eta_k\Big|=0\quad\text{in probability}.
\end{equation}

From now on, we choose any particular $M^\ast$ as above, for instance, $M^\ast(s)=\lfloor s^{1/2}\rfloor$ and put $a(s):=(\log M^\ast(s))^{1/2}$ for $s\geq 0$. Arguing as in the proof of Lemma \ref{lemma_3} or in the analysis of $I_{n,2}(s)$ in the proof of Lemma \ref{lemma5} we obtain
\begin{equation}\label{lem4.1112}
\lim_{s\to+\infty}\sup_{t\in [0,1]}\sum_{k\ge M^\ast(s)+1}\frac{(\log k)^{-1/2}}{k^{1/2+\exp(-ts)}}|\eta_k|\1_{\{|\eta_k|>k^{1/2}a(s)\}}=0\quad \text{\rm{a.s.}}
\end{equation}
and
\begin{equation}\label{lem4.2112}
\lim_{s\to\infty}\sup_{t\in [0,1]}\sum_{k\ge M^\ast(s)+1}\frac{(\log k)^{-1/2}}{k^{1/2+\exp(-ts)}}\mathbb{E}\big[|\eta_k|\1_{\{|\eta_k|>k^{1/2}a(s)\}}\big]=0.
\end{equation}

Put $\eta^\ast_k(s):=\eta_k\1_{\{|\eta_k|\leq k^{1/2}a(s)\}}-\me \big[\eta_k\1_{\{|\eta_k|\leq k^{1/2}a(s)\}}\big]$ for $k\in\mn$ and $s>0$. Put
\begin{equation*}
X^\ast(t,s):=\frac{1}{s^{1/2}}\sum_{k\ge M^\ast(s)+1}\frac{(\log k)^{-1/2} \eta^\ast_k(s)}{k^{1/2+\exp(-ts)}},\quad t\geq 0,~s>0.
\end{equation*}
In view of \eqref{eq:inter3}, \eqref{lem4.1112} and \eqref{lem4.2112} it remains to prove tightness of the distributions of $(X^\ast(t,s))_{t\in [0,1]}$ for large $s>0$. According to formula (7.8) on p.~82 in \cite{Billingsley:1999}, it is enough to show that, for all $\varepsilon>0$,
\begin{equation}\label{eq:tight}
\lim_{i\to\infty}\limsup_{s\to\infty}\mmp\{\sup_{u,v\in [0,1],\, |u-v|\leq 2^{-i}}\,|X^\ast(u,s)-X^\ast(v,s)|>\varepsilon\}=0.
\end{equation}
The proof of \eqref{eq:tight} follows closely the last part of the proof of Lemma \ref{lemma5}. We use dyadic partitions of $[0,1]$ by points $t^\ast_{j,\,m}:=2^{-j}m$ for $j\in\mn_0$ and $m=0,1,\ldots, 2^j$. Similarly to the argument preceding formula \eqref{5.12} we infer $$\sup_{u,v\in [0,1],\, |u-v|\leq 2^{-i}}\,|X^\ast(u,s)-X^\ast(v,s)|\leq \sum_{j\ge i}\max_{1\le m \le 2^j}|X^\ast(t^\ast_{j,\,m}, s)-X^\ast(t^\ast_{j,\,m-1},s)|.$$ Thus, it suffices to prove that, for all $\varepsilon>0$,
\begin{equation*}
\lim_{i\to\infty}\limsup_{s\to +\infty}\mathbb{P}\Big\{\sum_{j\ge i}\max_{1\le m \le 2^j}|X^\ast(t^\ast_{j,\,m},s)-X^\ast(t^\ast_{j,\,m-1},s)|>\varepsilon s^{1/2}\Big\}=0.
\end{equation*}
Put $a^\ast_j:=2^{-j/2}j^2$ for $j\in\mathbb{N}_0$. The last limit relation follows if we can show that, for all $\varepsilon>0$,
$$\lim_{i\to\infty}\limsup_{s\to +\infty}\sum_{j\ge i}\mathbb{P}\Big\{\max_{1\le m \le 2^j}|X^\ast(t^\ast_{j,\,m})-X^\ast(t^\ast_{j,\,m-1})|>\varepsilon a^\ast_j s^{1/2}\Big\}=0.$$
Denote by $A^\ast(j,s)$ and $B^\ast(u,k,s)$ the counterparts of $A(j,n)$ and $B(u,k,n)$ in the present situation. Then $A^\ast(j,s)\leq 2^{-j}s$ and, if $2^j\geq s$, $B^\ast(u,k,s)\leq |u|2^{-j}s=:C^\ast_j(u,s)$, if $2^j\leq s$, $B^\ast(j,k,s)\leq |u|=:C^\ast_j(u,s)$. With these at hand we obtain
\begin{equation*}
\me\big[\exp\big(\pm u(X^\ast(t^\ast_{j,\,m},s)-X^\ast(t^\ast_{j,\,m-1},s))\big)\big]\leq \exp\Big(\frac{2^{-j}su^2}{2}\eee^{C^\ast_j(u,s)}\Big),\quad u\in\mathbb{R}
\end{equation*}
and thereupon
\begin{multline*}
\mathbb{P}\big\{|X^\ast(t^\ast_{j,\,m},s)-X^\ast(t^\ast_{j,\,m-1},s)|>\varepsilon a^\ast_j s^{1/2} \big\}\leq \exp(-u \varepsilon a^\ast_j s^{1/2})\mathbb{E}[\exp(u|X^\ast(t^\ast_{j,\,m},s)-X^\ast(t^\ast_{j,\,m-1},s)|)]\\\leq  2\exp\Big(-u\varepsilon 2^{-j/2}j^2s^{1/2}+\frac{2^{-j}su^2}{2}\eee^{C^\ast_j(u,s)}\Big).
\end{multline*}
Put $u=\varepsilon 2^{j/2}s^{-1/2}$. Then $C^\ast_j(u,s)\leq \varepsilon$ and further
\begin{multline*}
\sum_{j\ge i}\mathbb{P}\Big\{\max_{1\le m \le 2^j}|X^\ast(t^\ast_{j,\,m},s)-X^\ast(t^\ast_{j,\,m-1},s)|>\varepsilon a^\ast_j s^{1/2}\Big\}\\ \le 2\exp(\varepsilon^2\eee^{\varepsilon}/2)\sum_{j\ge i}2^j \eee^{-\varepsilon^2 j^2}~\to~ 0,\quad i\to\infty.
\end{multline*}
The proof of Theorem \ref{thm:flt} is complete.

\section{A failure of approach based on a strong approximation}\label{sect:failure}

Our proof of Theorem \ref{main} is quite technical. Naturally we wanted to work out a less technical argument. One promising possibility has been to exploit a strong approximation result for centered standard random walks with finite variance, see, for instance, Theorem 12.6 in \cite{Kallenberg:1997}.
\begin{lemma}\label{lem:strong}
There exists a standard Brownian motion $(W(t))_{t\geq 0}$ such that $$T_{\lfloor t\rfloor}-W(t)=o((t\log\log t)^{1/2}),\quad t\to\infty\quad\text{{\rm a.s.}}$$
\end{lemma}

Implicit in the preceding proofs is the fact that the principal contribution to the LIL is made by the fragment of the original series which has the variance comparable to the variance of the full series. More precisely, one may reduce attention to the sum $\sum_{k=N_1(s)}^{N_2(s)}\frac{(\log k)^{-1/2}}{k^{1/2+s}}\eta_k$, where $N_1$ and $N_2$ are positive integer-valued functions satisfying $\lim_{s\to 0+}(\log\log N_1(s)/\log 1/s)=0$, $\lim_{s\to 0+}s\log N_2(s)=0$ and
\begin{equation}\label{eq:inter2}
\lim_{s\to 0+}(\log\log N_2(s)/\log 1/s)=1.
\end{equation}
The reason is that $${\rm Var}\,\Big(\sum_{k=N_1(s)}^{N_2(s)}\frac{(\log k)^{-1/2}}{k^{1/2+s}}\eta_k\Big)~\sim~{\rm Var}\,\Big(\sum_{k\geq 2}\frac{(\log k)^{-1/2}}{k^{1/2+s}}\eta_k\Big)~\sim~\log 1/s,\quad s\to 0+.$$

We hoped it would be possible to work with Gaussian random variables given by \newline $\int_{N_1(s)}^{N_2(s)}\frac{(\log x)^{-1/2}}{x^{1/2+s}}{\rm d}W(x)$ in place of $\int_{N_1(s)}^{N_2(s)}\frac{(\log x)^{-1/2}}{x^{1/2+s}} {\rm d}T_{\lfloor x\rfloor}$. Unfortunately, Lemma \ref{lem:strong} does not seem to secure such a possibility. Indeed, after integrating by parts we intended to show that
$$\lim_{s\to 0+}f(s)\int_{N_1(s)}^{N_2(s)}\frac{|T_{\lfloor x\rfloor}-W(x)|}{(\log x)^{1/2} x^{3/2+s}}{\rm d}x=0\quad\text{a.s.}$$ In view of Lemma \ref{lem:strong} the latter would be a consequence of
\begin{equation}\label{eq:inter}
\lim_{s\to 0+}f(s)\int_{N_1(s)}^{N_2(s)}\frac{(\log\log x)^{1/2}}{(\log x)^{1/2} x^{1+s}}{\rm d}x=0.
\end{equation}
Since
\begin{multline*}
\int_{N_1(s)}^{N_2(s)}\frac{(\log\log x)^{1/2}}{(\log x)^{1/2} x^{1+s}}{\rm d}x=\frac{1}{s^{1/2}}\int_{s\log N_1(s)}^{s\log N_2(s)}\frac{(\log 1/s+\log z)^{1/2}\eee^{-z}}{z^{1/2}}{\rm d}z\\\leq \frac{(\log\log N_2(s))^{1/2}}{s^{1/2}}\int_0^{s\log N_2(s)} \frac{\eee^{-z}}{z^{1/2}}{\rm d}z~\sim~ 2 (\log N_2(s)\log\log N_2(s))^{1/2},\quad s\to 0+, 
\end{multline*}
the validity of \eqref{eq:inter} required $\lim_{s\to 0+}(\log N_2(s)/\log 1/s)=0$, which was incompatible with \eqref{eq:inter2}.

We note in passing that another strong approximation of $(T_n)_{n\geq 0}$, this time by sums of independent Gaussian random variables, has a better error 
the big $O$ of square root, see \cite{Major:1979}. Revisiting the calculation above reveals that this approximation does not seem to help either.

\noindent \textbf{Acknowledgment.} The present work was supported by the National Research Foundation of Ukraine (project 2023.03/0059 ‘Contribution to modern theory of random series’).

\end{document}